\documentclass[12pt]{amsart}
\usepackage{amsmath}
\usepackage{amstext}
\usepackage{amsfonts}
\usepackage{amssymb}
\usepackage{amsthm}
\usepackage{amsrefs}
\allowdisplaybreaks

\usepackage{microtype}
\usepackage{color,hyperref}

\definecolor{darkblue}{rgb}{0.0,0.0,0.3}
\hypersetup{colorlinks,breaklinks, 
            linkcolor=red,urlcolor=red,
            anchorcolor=red,citecolor=red}

\theoremstyle{plain}
\newtheorem{thm}{Theorem}[section]
\newtheorem{conj}[thm]{Conjecture}
\newtheorem{cor}[thm]{Corollary}
\newtheorem{prop}[thm]{Proposition}
\newtheorem{lem}[thm]{Lemma}

\theoremstyle{definition}
\newtheorem{defn}[thm]{Definition}
\newtheorem{example}[thm]{Example}

\newtheorem{rem}[thm]{Remark}
\newtheorem{question}[thm]{Question}

\numberwithin{equation}{section}

\newcommand{\bC}{{\mathbb{C}}}
\newcommand{\bR}{{\mathbb{R}}}

\newcommand{\B}{\mathcal{B}}

\newcommand{\G}{\mathcal{G}}
\newcommand{\J}{\mathcal{J}}
\newcommand{\M}{\mathcal{M}}
\renewcommand{\O}{\mathcal{O}}

\newcommand{\rC}{\mathrm{C}}
\newcommand{\ep}{\varepsilon}

\newcommand{\ca}{\mathrm{C}^*}

\newcommand{\diag}{\operatorname{diag}}
\newcommand{\Int}{\operatorname{int}}
\newcommand{\OP}{\operatorname{OP}}
\newcommand{\supp}{\operatorname{supp}}
\newcommand{\ip}[1]{\langle #1 \rangle}
\newcommand{\bip}[1]{\big\langle #1 \big\rangle}
\newcommand{\ol}{\overline}

\newcommand{\qand}{\quad\text{and}\quad}

\newcommand{\qfor}{\quad\text{for}\quad}
\newcommand{\qforal}{\quad\text{for all}\quad}
\newcommand{\qif}{\quad\text{if}\quad}
\newcommand{\AND}{\ \text{and}\ }

\begin{document}
\title[Choquet order and hyperrigidity for function systems]{Choquet order and hyperrigidity\\ for function systems}

\author[K.R. Davidson]{Kenneth R. Davidson}
\address{Department of Pure Mathematics\\ University of Waterloo\\Waterloo, ON, N2L 3G1, Canada}
\email{krdavids@uwaterloo.ca}

\author[M. Kennedy]{Matthew Kennedy}
\address{Department of Pure Mathematics\\ University of Waterloo\\Waterloo, ON, N2L 3G1, Canada}
\email{matt.kennedy@uwaterloo.ca}

\begin{abstract}
We establish a dilation-theoretic characterization of the Choquet order on the space of measures on a compact convex set using ideas from the theory of operator algebras. This yields an extension of Cartier's dilation theorem to the non-separable case, as well as a non-separable version of \u Sa\u skin's theorem from approximation theory. We show that a slight variant of this order characterizes the representations of a commutative C*-algebras that have the unique extension property relative to a set of generators. This reduces the commutative case of Arveson's hyperrigidity conjecture to the question of whether measures that are maximal with respect to the classical Choquet order are also maximal with respect to this new order. An example shows that these orders are not the same in general.
\end{abstract}

\subjclass[2010]{Primary 46A55, 46L05; Secondary 41A36, 47A20, 47A58, 47L25}
\keywords{Choquet theory, hyperrigidity, Korovkin sets, completely positive maps}
\thanks{First author supported by NSERC Grant Number 3488-2013.}
\thanks{Second author supported by NSERC Grant Number 418585.}
\maketitle

\section{Introduction} \label{sec:introduction}

Choquet theory is an important subject within convex analysis that has applications to many different fields in mathematics. The most well-known result in Choquet theory is Choquet's theorem, which asserts that for a metrizable compact convex set $K$ and a point $x \in K$, there is always a probability measure on $K$ with barycenter $x$ that is supported on the set $\partial K$ of extreme points of $K$, in the sense that $\mu(\partial K) = 1$. 

If $K$ is non-metrizable, then $\partial K$ is not necessarily Borel, in which case it no longer makes sense to say that $\mu$ is supported on $\partial K$. Nevertheless, Bishop and de Leeuw were able to extend Choquet's theorem to the non-metrizable setting. They considered a partial order on the set of measures on $K$ with the property that a maximal measure $\mu$ has essentially all of the desirable properties of a measure supported on $\partial K$. In particular, $\mu$ is almost supported on $\partial K$ in a suitable sense. Bishop and de Leeuw's theorem asserts that for a point $x \in K$, there is always a maximal probability measure $\mu$ on $K$ with barycenter $x$.

This idea that a measure-theoretic condition can be translated into an order-theoretic condition has become a central component of the modern-day approach to Choquet theory. The Choquet order on the set of measures on a compact convex set, which refines the partial order introduced by Bishop and de Leeuw, plays a key role in the theory.

For measures $\mu$ and $\nu$ on a compact convex set $K$, $\mu$ is said to be dominated by $\nu$ in the Choquet order, written $\mu \prec_c \nu$, if  $\mu(f) \leq \nu(f)$ for all continuous convex functions $f \in \rC(K)$. 

In this paper, we introduce two new partial orders on the space of measures based on ideas from the theory of operator algebras and dilation theory. The first, which we call the strong dilation order, turns out to be equivalent to the Choquet order. This equivalence leads to a new proof of Cartier's theorem, which asserts that if $K$ is metrizable and $\mu$ and $\nu$ are measures on $K$ with $\mu$ dominated by $\nu$ in the Choquet order, then there is a family of probability measures $(\lambda_x)_{x \in K}$ with the property that
\[ \int f \,d\nu = \int_K \lambda_x(f) \,d\mu \qforal f \in \rC(K) . \]

While existing proofs of Cartier's theorem rely on the theory of disintegration of measures, which accounts for the requirement that $K$ be metrizable, our proof instead uses the equivalence of the Choquet order and the strong dilation order, in addition to a lifting theorem of Maharam. As a consequence, our proof does not require the metrizability of $K$.

A second application of the equivalence of the Choquet order and the strong dilation order is to approximation theory. A classical theorem of Korovkin shows that if $\phi_n$ is a sequence of positive maps from $\rC[a,b]$ to itself satisfying $\lim_{n} \phi_n(x^j) = x^j$ uniformly for $j=0,1,2$, then $\lim_{n} \phi_n(f) = f$ uniformly for each $f \in \rC[a,b]$.

A result of \u Sa\u skin provides a major generalization of Korovkin's theorem. For a a compact metrizable space $X$, \u Sa\u skin's result provides necessary and sufficient conditions for a set of continuous functions on $\rC(X)$ to serve as a set of test functions for a sequence $\phi_n$ of positive maps from $\rC(X)$ into itself, in the sense that if $\phi_n(g) \to g$ uniformly for each test function, then $\phi_n$ converges pointwise to the identity on all of $\rC(X)$, meaning that $\phi_n(f) \to f$ uniformly for each $f\in \rC(X)$.

We give a new proof of \u Sa\u skin's theorem using the equivalence of the Choquet order with the strong dilation order. Once again, our proof does not require the metrizability of $X$, provided that sequences are replaced by nets.

The second partial order on measures, which we call the dilation order, is slightly weaker than the strong dilation order, in the sense that domination in the strong dilation order implies domination in the dilation order. As a consequence, measures that are maximal with respect to the dilation order are also maximal with respect to the Choquet order. 

We show that the dilation order characterizes representations of unital commutative C*-algebras with the unique extension property relative to a generating function system. In particular, a measure $\mu$ on a compact convex set $K$ is maximal with respect to the dilation order if and only if the corresponding representation $\pi_\mu$ of $\rC(K)$ on $L^2(\mu)$ has the unique extension property relative to the space $A(K)$ of continuous affine functions on $K$, meaning that the restriction $\pi_\mu|_{A(K)}$ has a unique extension to a completely positive map on $\rC(K)$. 

This has applications to an operator-valued extension of the theorem of Korovkin and \u Sa\u skin conjectured by Arveson. Arveson says that a unital subspace $A$ of a C*-algebra $C$ is hyperrigid relative to $C$ if for every representation $\pi$ of a separable C*-algebra $C$ on a Hilbert space $H$, and every net $\phi_n$ of completely positive maps from $C$ to $\B(H)$ such that $\lim_{n} \phi_n(a) = \pi(a)$ for all $a \in A$, then necessarily $\lim_{n} \phi_n(c) = c$ for all $c \in C$. 

Arveson showed that $A$ is hyperrigid relative to $C$ if and only if every representation of $C$ has the unique extension property relative to $A$. Arveson's hyperrigidity conjecture, which has now become an important open problem, is the assertion that $A$ is hyperrigid in $C$ if and only if every irreducible representation of $C$ has the unique extension property relative to $A$.

In the commutative setting, Arveson's hyperrigidity conjecture reduces to the question of the hyperrigidity of the space $A(K)$ of continuous affine functions on a compact convex set $K$ relative to the C*-algebra $\rC(\ol{\partial K})$ of continuous functions on the closure of the extreme points of $K$.

Our results imply $A(K)$ is hyperrigid relative to $\rC(K)$ if and only if the set $\partial K$ of extreme points of $K$ is closed and every measure supported on $\partial K$ is maximal with respect to the dilation order. In particular, a resolution of Arveson's conjecture in the commutative setting reduces to the question of whether maximality in the Choquet order implies maximality in the dilation order.

Finally, we give an example showing that the Choquet order and the dilation order do not coincide in general.

\section{Preliminaries} \label{sec:preliminaries}

Since this work makes considerable use of both Choquet theory and the theory of operator algebras, we will be somewhat generous in providing background material in both areas for the convenience of our readers.

\subsection{Commutative C*-algebras} \label{sec:comm-c-star-alg}

Let $X$ be a compact Hausdorff space and let $\rC(X)$ denote the C*-algebra of continuous functions on $X$. A linear functional $\alpha : \rC(X) \to \bC$ is said to be unital if $\alpha(1) = 1$ and positive if $\alpha(f) \geq 0$ for every non-negative function $f \in \rC(X)$. If $\alpha$ is both unital and positive, then it is said to be a {\em state}.

Let $M^+(X)$ denote the space of positive regular Borel measures on $X$, and let $P(X)$ denote the space of regular Borel probability measures on $X$. By the Riesz-Markov-Kakutani repre\-sen\-ta\-tion theorem, positive linear functionals on $\rC(X)$ correspond to measures in $M^+(X)$, and in particular, states on $\rC(X)$ correspond to probability measures in $P(X)$.

For $\mu \in M^+(X)$, the corresponding positive linear functional on $\rC(X)$ is defined by
\[
\mu(f) = \int_X f\, d\mu  \qfor f \in \rC(X).
\]

\begin{defn} \label{D:representation}
A {\em repre\-sen\-ta\-tion of $\mu$} is a tuple $(\pi, H, \xi)$ consisting of a $*$-repre\-sen\-ta\-tion $\pi : \rC(X) \to \B(H)$ on a Hilbert space $H$ and a distinguished vector $\xi \in H$ such that
\[
\mu(f) = \ip{\pi(f)\xi, \xi} \qforal f \in \rC(X).
\]
\end{defn} 

We will write $(\pi_\mu, L^2(\mu), 1_\mu)$ for the repre\-sen\-ta\-tion obtained from the Gelfand-Naimark-Segal (GNS) construction for the positive functional $\mu$, where $L^2(\mu) = L^2(X,\mu)$, $1_\mu$ denotes the constant function $1$ considered as an element of $L^2(\mu)$, and $\pi_\mu : \rC(X) \to \B(L^2(\mu))$ is defined by
\[
\pi_\mu(f) h = fh \qfor f \in \rC(X),\ h \in L^2(\mu).
\]
The GNS repre\-sen\-ta\-tion of $\mu$ is minimal, in the sense that if $(\pi, H, \xi)$ is another repre\-sen\-ta\-tion of $\mu$, then the restriction of $\pi$ to the cyclic invariant subspace for $\pi$ generated by $\xi$ is unitarily equivalent to $\pi_\mu$ via a unitary that maps $\xi$ to $1_\mu$.

It is a standard fact from the theory of representations of C*-algebras that every $*$-representation $\pi : \rC(X) \to \B(H)$ can be written as a direct sum of cyclic $*$-representations. Furthermore, every cyclic $*$-representation is unitarily equivalent to the GNS representation $\pi_\mu$ for some measure $\mu \in M^+(X)$. 

For a compact subset $C \subset X$, we will say that a $*$-representation $\pi$  of $\rC(X)$ is {\em supported on $C$} if there are measures $(\mu_i)_{i \in I}$ in $M^+(X)$ such that $\pi$ is unitarily equivalent to the direct sum $\oplus_{i \in I} \pi_{\mu_i}$, and each $\mu_i$ is supported on $C$.

\subsection{Function systems} \label{sec:function-systems}

The notion of a function system was introduced by Kadison in \cite{Kad1951}. An {\em abstract function system} $A$ is an ordered normed vector space that is positively generated, i.e. $A = A^+ - A^+$, and has a distinguished archimedean order unit $1_A$ such that the norm on $A$ is determined via the formula
\[
\|a\| = \inf \{ \lambda > 0 : -\lambda 1_A \leq a \leq \lambda 1_A \} \qfor a \in A.
\]
We will consider function systems over the complex numbers. Although the literature often considers function systems over the real numbers, results in the complex case are readily derived from the real case.

A linear functional $\alpha : A \to \bC$ is said to be unital if $\alpha(1_A) = 1$ and positive if $\alpha(A^+) \subset \bR^+$. If $\alpha$ is both unital and positive, then it is said to be a state. The {\em state space} $S(A)$ of $A$ is the compact convex space of states on $A$ equipped with the weak-$*$~topology.

If $B$ is another function system, then a map $\phi : A \to B$ is said to be unital if $\phi(1_A) = 1_B$, and is said to be positive if $\phi(A^+) \subset B^+$. If $\phi$ is bijective, then it is said to be an {\em order isomorphism} if it is unital and both $\phi$ and $\phi^{-1}$ are positive.

A {\em concrete function system} is a unital self-adjoint subspace of a unital commutative C*-algebra. Observe that a concrete function system, considered as a vector space over the real numbers, is an abstract function system in the above sense. By Kadison's repre\-sen\-ta\-tion theorem, every abstract function system is order isomorphic to a canonical concrete function system. We collect this result, as well as several closely related results on function systems in the next theorem. For details we refer the reader to the book of Alfsen and Shultz \cite{AS2001}.

\begin{thm}[Kadison] \label{thm:function-systems}
Let $A$ be a function system with state space $K := S(A)$. Then $A$ is order isomorphic to a dense subspace of the space $A(K)$ of continuous affine function on $K$ via the map $\iota: A \to A(K) : a \to \hat{a}$, where
\[
\hat{a}(\alpha) = \alpha(a) \qfor a \in A,\ \alpha \in K.
\]
Moreover, if $A$ is a concrete function system that generates a commutative C*-algebra $\rC(X)$, then there is a $*$-homomorphism $q : \rC(K) \to \rC(X)$ such that $q \circ \iota$ is the identity on $A$.
\end{thm}

Observe that if $A$ is complete, then Theorem \ref{thm:function-systems} implies in particular that $A$ is order isomorphic to $A(K)$. We will  work with function systems that are complete in this paper.

Theorem \ref{thm:function-systems} largely reduces the study of abstract function systems to the study of the concrete function systems of continuous affine functions on compact convex sets.

\subsection{Choquet boundary} \label{sec:choquet-boundary}

For an overview of Choquet theory, we refer the reader to the books of Alfsen \cite{Alf} and Phelps \cite{Phelps}.

Let $A$ be a concrete function system that generates a commutative C*-algebra $\rC(X)$. Let $K = S(A)$ denote the state space of $A$, and let $\iota : A \to A(K)$ and $q : \rC(K) \to \rC(X)$ be as in Theorem \ref{thm:function-systems}. Then letting $q^* : \rC(X)^* \to \rC(K)^*$ denote the adjoint of $q$ and identifying points in $X$ and $K$ with the corresponding point evaluations, $q^*$ maps $X$ into $K$. The {\em Choquet boundary} $\partial_A X$ of $A$ is $\partial_A X = (q^*)^{-1}(\partial K)$ where $\partial K$ denotes the set of extreme points of $K$. Note that $q^*(\partial_A X) = \partial K$.

In particular, observe that the Choquet boundary $\partial_{A(K)} K$ of $A(K)$ is precisely the set $\partial K$ of extreme points of $K$.

\subsection{Choquet order}

Let $K$ be a compact convex subset of a locally convex vector space and let $\mu \in M^+(K)$ be a positive measure. If $K$ is metrizable, then the set $\partial K$ of extreme points of $K$ is a $G_\delta$ set. In this case, we will say that $\mu$ is {\em supported} on $\partial K$ if $\mu(K \setminus \partial K) =0$.

If $K$ is non-metrizable, then Bishop-de Leeuw \cite{BdL} showed that $\partial K$ is not necessarily even Borel. In general, we will say that $\mu$ is {\em pseudo-supported} on $\partial K$ if $\mu(X) = 0$ for every Baire subset $X \subset K$ with $X \cap \partial K = \emptyset$. Recall that a subset of $K$ is a Baire set if it belongs to the $\sigma$-algebra generated by all compact $G_\delta$ subsets of $K$. If $K$ is metrizable, then every closed subset is a $G_\delta$. Thus, in this case, $\mu$ is pseudo-supported on $\partial K$ if and only if it is supported on $\partial K$.  In the general case, one can at least assert that a measure which is pseudo-supported on $\partial K$ is supported on $\ol{\partial K}$.

\begin{defn}[Choquet order]
Let $K$ be a compact convex subset of a locally convex vector space. The {\em Choquet order} ``$\prec_c$'' on $M^+(K)$ is defined for $\mu,\nu \in M^+(K)$ by $\mu \prec_c \nu$ if $\mu(f) \leq \nu(f)$ for every continuous convex function $f \in \rC(K)$.
\end{defn}

\begin{defn}[Boundary measure]
Let $K$ be a compact convex subset of a locally convex vector space. A measure in $M^+(K)$ is said to be a {\em boundary measure} if it is maximal in the Choquet order.
\end{defn}

The next result combines \cite{Alf}*{Prop.I.4.5, Cor.I.4.12}.

\begin{thm}[Mokobodzki, Bishop-de Leeuw]
Let $K$ be a compact convex subset of a locally convex vector space. Then every boundary measure in $M^+(K)$ is pseudo-supported on $\partial K$. In particular if $\partial K$ is closed, then every boundary measure is supported on $\partial K$.  If $K$ is metrizable, then conversely, every measure in $M^+(K)$ that is supported on $\partial K$ is a boundary measure.
\end{thm}

For a point $x \in K$, let $\delta_x$ denote the corresponding point mass. The set $P_x(K) := \{\mu \in P(K) \mid \delta_x \prec_c \mu\}$ is precisely the set of probability measures on $K$ that represent $x$, in the sense that for $\mu \in P_x(K)$, $\mu(x) = a(x)$ for all $a \in A(K)$.

\subsection{Cartier's theorem} \label{sec:cartier}

In classical Choquet theory, there is a notion of {\em dilation of measures} that, at first glance, appears to have little in common with the notion of dilation arising in the theory of completely positive linear maps on C*-algebras. Cartier's theorem characterizes the Choquet order for metrizable compact convex sets in terms of dilation of measures.

\begin{thm}[Cartier] \label{thm:cartier}
Let $K$ be a metrizable compact convex subset of a locally convex vector space. Let $\mu,\nu \in M^+(K)$ satisfy $\mu \prec_c \nu$. Then $\mu$ is dilated by $\nu$, meaning there is a family $\{\lambda_x\}_{x \in K} \subset P(K)$ of probability measures such that
\begin{enumerate}
\item $\lambda_x \in P_x(K)$ for $\mu$-a.e. $x \in K$,
\item $f \to \lambda_x(f)$ is $\mu$-measurable for all $f \in \rC(K)$, and
\item $\int_K f\, d\nu = \int_K \lambda_x(f)\, d\mu$ for all $f \in \rC(K)$.
\end{enumerate}
\end{thm}

In the metrizable setting, Cartier's theorem can be used to give a short proof of one direction of the equivalence between Choquet order and the strong dilation order. In the non-metrizable setting, where we are unable to make use of Cartier's theorem, we will need to work much harder to prove this result. However, our methods will yield an extension of Cartier's theorem to the non-metrizable setting, which seems to have been an open problem for a considerable time (see e.g. \cite{Eff1971}). Moreover it provides a new proof of Cartier's theorem in the metrixable case which does not use disintegration of measures.

\section{The Choquet order} \label{sec:strong-dilation-order}

Let $\mu$ be a measure on $K$.
We consider the algebra $L^\infty(\mu)$ as embedded into $\B(L^2(\mu))$ as the multiplication operators.
The following definition was motivated by the dilation order, which we will define later. In subsequent sections we will see that this notion relates to Cartier's theorem and the theory of liftings.

\begin{defn} \label{defn:strong-dilation-order}
Let $K$ be a compact convex subset of a locally convex vector space. 
The {\em strong dilation order} ``$\prec_s$'' on $M^+(K)$ is defined for $\mu,\nu \in M^+(K)$ by $\mu \prec_s \nu$ if there is a positive map $\Phi:\rC(K) \to L^\infty(\mu)$ such that 
\[ \Phi(a) = \pi_\mu(a) \qforal a \in A(K) ,\]
and
\[ \nu(f) = \ip{\Phi(f) 1_\mu, 1_\mu} \qforal f \in \rC(K).\]
\end{defn}

We will show that the strong dilation order is equivalent to the Choquet order. The next result is the easy direction.

\begin{thm}\label{T:strong_diln_order_implies_Choquet_order}
Let $K$ be a compact convex subset of a locally convex vector space, and let $\mu,\nu \in M^+(K)$. 
Then $\mu \prec_s \nu$ implies that $\mu \prec_c \nu$.
\end{thm}

\begin{proof}
Suppose that $\mu \prec_s \nu$. 
Let $\Phi:\rC(K) \to L^\infty(\mu)$ be a positive map such that $\Phi|_{A(K)} = \pi_\mu|_{A(K)}$ and $\nu(f) = \ip{\Phi(f) 1_\mu, 1_\mu}$ for all $f\in\rC(K)$.

Fix a convex function $f\in P(K)$. Then 
\[ f = \sup \{ a \in A(K) : a \le f \} . \]
Since $L^\infty(\mu)$ is a commutative von Neumann algebra, we can take supremums in this algebra.
Since $\Phi$ is positive,  we obtain
\begin{align*}
 \pi_\mu(f) &= \sup \{ \pi_\mu(a) : a \in A(K),\ a \le f \}\\
 &= \sup \{ \Phi(a) : a \in A(K),\ a \le f \}\\
 &\le  \Phi(f).
\end{align*}
Now we can compute
\begin{align*}
 \mu(f) &= \ip{ \pi_\mu(f) 1_\mu, 1_\mu} 
 \le  \ip{ \Phi(f) 1_\mu, 1_\mu} = \nu(f) .
\end{align*}
This shows that $\mu \prec_c \nu$.
\end{proof}

For the other direction, we will first give a short proof in the metrizable case using Cartier's theorem from Section \ref{sec:cartier}. The apparent simplicity of the proof is deceptive, since the proof of Cartier's theorem requires a significant amount of work using the theory of disintegration of measures.

\begin{thm} \label{T:dilation via dilation}
Let $K$ be a metrizable compact convex subset of a locally convex space, and let $\mu,\nu \in M^+(K)$. Then $\mu \prec_c \nu$ implies $\mu \prec_s \nu$.
\end{thm}

\begin{proof}
By Theorem \ref{thm:cartier} (Cartier's theorem), there is a family of probability measures $\{\lambda_x : x \in K\}$ such that $\delta_x \prec_c \lambda_x$ for $\mu$-a.e.\,$x \in K$, the map $f \to \lambda_x(f)$ is $\mu$-measurable and $\nu(f) = \int_K \lambda_x(f)\, d\mu$ for all $f \in \rC(K)$.

Define a positive unital map $\Phi:\rC(K) \to L^\infty(\mu)$ by $\Phi(f)(x) = \lambda_x(f)$, and identify functions in $L^\infty(\mu)$ with the corresponding multiplication operators in $\B(L^2(\mu))$. Then for $a \in A(K)$, we have $\Phi(a)(x) = a(x)$ $\mu$-a.e., giving $\Phi(a) = \pi_\mu(a)$. Now for $f\in\rC(K)$,
\[ \ip{ \Phi(f) 1_\mu, 1_\mu} = \int_K \lambda_x(f) \,d\mu = \nu(f) .\]
Therefore $\mu \prec_s \nu$.
\end{proof}

We now present a proof that is valid in the non-metrizable case, where Cartier's theorem does not apply. The proof utilizes operator-algebraic methods, and does not require the theory of disintegration of measures. In fact, as a consequence of our approach, we will obtain an extension of Cartier's theorem to the non-metrizable setting.

We begin with a lemma which is readily obtained by taking a direct sum of two repre\-sen\-ta\-tions.

\begin{lem} \label{lem:dilation-order-additive}
Let $K$ be a compact convex subset of a locally convex vector space, and let $\mu_1,\mu_2,\nu_1,\nu_2 \in M^+(K)$ satisfy  $\mu_1 \prec_s \nu_1$ and $\mu_2 \prec_s \nu_2$. Then $\mu_1 + \mu_2 \prec_s \nu_1 + \nu_2$.
\end{lem}

We first handle the case of an atomic measure. 

\begin{prop} \label{choquet-implies-dilation-discrete}
Let $K$ be a compact convex subset of a locally convex vector space, and let $\mu,\nu \in M^+(K)$. 
Suppose that $\mu$ is atomic, and thus has finite or countable support. Then $\mu \prec_c \nu$ implies $\mu \prec_s \nu$.
\end{prop}

\begin{proof}
Since $\mu$ is atomic, we can write it as $\mu = \sum_{i=1}^\kappa \alpha_i \delta_{x_i}$ for positive real numbers $(\alpha_i)_{i=1}^\kappa$ summing to one and points $(x_i)_{i=1}^\kappa$ in $K$, where $\kappa \in \mathbb{N} \cup \{\infty\}$. By repeated use of the Cartier-Fell-Meyer theorem \cite{Alf}*{Proposition I.3.2}, we obtain measures $(\nu_i)_{i=1}^\kappa$ in $P(K)$ such that 
\[ \delta_{x_i} \prec_c \nu_i \qand  \nu = \sum_{i=1}^\kappa \alpha_i \nu_i . \]

The GNS repre\-sen\-ta\-tion $\pi_\mu : \rC(K) \to B(L^2(\mu))$ of $\mu$ can be written as
\[
\pi_\mu(f) = \sum_{i=1}^\kappa f(x_i) \chi_{x_i} \qfor f \in \rC(K).
\]
Define $\Phi : \rC(K) \to L^\infty(\mu)$ by
\[
\Phi(f) = \sum_{i=1}^\kappa \nu_i(f) \chi_{x_i} \qfor f \in \rC(K).
\]
Then $\Phi(a) = \pi_\mu(a)$ for all $a \in A(K)$, and an easy computation shows 
\[ \ip{ \Phi(f) 1_\mu, 1_\mu } = \sum_{i=1}^\kappa \alpha_i \nu_i(f) = \nu(f) .\]
Therefore, $\mu \prec_s \nu$.
\end{proof}

If $K$ is a compact convex subset of a locally convex vector space, and $\mu,\nu \in M^+(K)$ satisfy $\mu \prec_c \nu$, then we can decompose $\mu = \mu_d + \mu_c$ into its atomic and continuous parts. The Cartier-Fell-Meyer theorem \cite{Alf}*{Proposition I.3.2} implies that we can decompose $\nu$ as $\nu = \nu_1 + \nu_2$ for $\nu_1,\nu_2 \in M^+(K)$ satisfying $\mu_d \prec_c \nu_1$ and $\mu_c \prec_c \nu_2$. Proposition~\ref{choquet-implies-dilation-discrete} yields $\mu_s \prec_s \nu_1$. Thus by Lemma~\ref{lem:dilation-order-additive}, it remains to deal with the continuous part. This is accomplished by approximating $\mu_c$ in an appropriate way by atomic measures.

Consider the following set of pairs of probability measures that are comparable in the Choquet order:
\[ M_c := \{ (\mu, \nu) : \mu \prec_c \nu,\ \mu,\nu \in P(K) \} .\]
By \cite{Alf}*{Lemma I.3.7}, $M_c$ is a weak-$*$ compact convex subset of the product $P(K) \times P(K)$, and the extreme points of $M_c$ are contained in $S = \bigcup_{x\in K} \{\delta_x\}\times P_x(K)$. In particular, it follows from the Krein-Milman theorem that the convex hull of $S$ is weak-$*$ dense in $M_c$.

We require the following technical approximation result.

\begin{lem} \label{L:technical lemma}
Let $\mu,\nu \in P(K)$ be probability measures such that $\mu$ is continuous and $\mu \prec_c \nu$. Let $E \subset A(K)$, $F \subset \rC(K)$ and  $\Xi \subset L^2(\mu)$ be finite sets. Then for $\ep > 0$, there is a unital positive map $\Phi:\rC(K) \to L^\infty(\mu)$ satisfying
\[ \big| \ip{\Phi(a) \xi, \eta} - \ip{\pi_\mu(a) \xi, \eta} \big| < \ep \qforal a \in E \ \text{and all}\ \ \xi, \eta \in \Xi, \]
and
\[ \big| \ip{\Phi(f) 1_\mu, 1_\mu} - \nu(f) \big| < \ep \qforal f \in F .\]
\end{lem}

\begin{proof}
Let $N = \max\{ 1,\ \|a\|_\infty,\  \|\xi\|^2 : a \in E,\ \xi \in \Xi\}$, and set $\ep' = \ep/2N$. Choose $\delta > 0$ so that if $\mu(C) < \delta$, then 
\[ \max\{ \| \chi_C \xi\|^2 : \xi \in \Xi \} < \frac{\ep'}2 = \frac\ep{4N}.\]
By the uniform continuity of $a \in E$ and the compactness of $K$, we may find a finite collection of compact convex sets $K_i \subset K$ for $1 \le i \le n$ such that their interiors cover $K$ and 
\[ |a(x)-a(y)| < \ep' \qforal a \in E \AND x,y \in K_i,\ 1 \le i \le n .\]
Choose Borel sets $B_i \subset \ol{B_i} \subset \Int(K_i)$ that partition $K$ into a disjoint union of sets.
Define $\mu_i = \mu|_{B_i}$. 
Applying the Cartier-Fell-Meyer theorem \cite{Alf}*{Proposition I.3.2} yields measures $\nu_i$ such that $\mu_i \prec_c \nu_i$ and $\sum_{i=1}^n \nu_i = \nu$.
By Urysohn's lemma, there are functions $g_i$ in $\rC(K)$ such that $\chi_{B_i} \le g_i \le \chi_{K_i}$. 

Applying \cite{Alf}*{Lemma I.3.7} as mentioned above to $(\mu_i,\nu_i)$, we can find positive measures $(\sigma_i, \tau_i)$ in the convex hull of $S$ such that $\sigma_i$ is finitely supported and the pair $(\sigma_i,\tau_i)$ approximates $(\mu_i,\nu_i)$. Specifically, there are constants $\alpha_{ik} > 0$, points $x_{ik} \in K$ and probability measures $\tau_{ik} \in P_{x_{ik}}(K)$ such that
\[ \sum_{k=1}^{m_i} \alpha_{ik} = \|\mu_i\|, \quad \sigma_i = \sum_{k=1}^{m_i} \alpha_{ik} \delta_{x_{ik}}, \quad \tau_i = \sum_{k=1}^{m_i} \alpha_{ik} \tau_{ik}, \]
and such that $\sigma_i$ approximates $\mu_i$ and $\tau_i$ approximates $\nu_i$, meaning
\[ |\sigma_i(g_i) - \mu_i(g_i)| < \delta \|\mu_i\| \qfor 1 \le i \le n,  \]
and
\[ |\tau_i(f) - \nu_i(f)| < \ep \|\mu_i\| \qforal f \in F \AND 1 \le i \le n. \]

Notice that
\[ \|\mu_i\|=\|\nu_i\|=\|\sigma_i\|=\|\tau_i\| =  \sum_{k=1}^{m_i} \alpha_{ik}. \]
The condition  
\[ \delta \|\mu_i\| > |\sigma_i(g_i) - \mu_i(g_i)|  = \|\mu_i\| - \sum_{k=1}^{m_i} \alpha_{ik} g(x_{ik})  \]
shows that the set $\G_i = \{k : x_{ik} \in K_i \}$ is sufficiently large that
\[\sum_{k\in\G_i} \alpha_{ik} > (1 - \delta) \|\mu_i\|. \]
This means that most of the mass of $\sigma_i$ is supported on $K_i$. 

Partition each $B_i$ into disjoint Borel sets $B_{ik}$ so that $\mu_i(B_{ik}) = \alpha_{ik}$. 
Define $C = \bigcup_{i=1}^n \bigcup_{k \not\in\G_i} B_{ik}$. Then
\[ \mu(C) = \sum_{i=1}^n \sum_{k \not\in\G_i} \alpha_{ik} < \delta \sum_{i=1}^n \|\mu_i\| = \delta .\] 
Define a positive map $\Phi:\rC(K) \to L^\infty(\mu)$ by
\[ \Phi(f) = \sum_{i=1}^n \sum_{k=1}^{m_i} \tau_{ik}(f) \chi_{B_{ik}} \qfor f \in \rC(K).\]
Identify $L^\infty(\mu)$ with the corresponding multiplication operators on $\B(L^2(\mu))$, so that the range of $\Phi$ is contained in $\B(L^2(\mu))$.
Evidently $\Phi$ is positive, and since each $\tau_{ik}$ is a probability measure, it is clear that $\Phi(1) = 1$.
 
To verify the second inequality, we compute
\begin{align*}
\ip{ \Phi(f) 1_\mu, 1_\mu } &= \sum_{i=1}^n \sum_{k=1}^{m_i} \tau_{ik}(f) \int_K \chi_{B_{ik}}\, d\mu \\
&= \sum_{i=1}^n \sum_{k=1}^{m_i} \alpha_{ik}\tau_{ik}(f)  \\
&= \sum_{j=1}^n \tau_i (f)  . 
\end{align*}
Therefore for $f \in F$, we have
\begin{align*}
 \big| \ip{\Phi(f) 1_\mu, 1_\mu} - \nu(f) \big| &= \Big| \sum_{j=1}^n \tau_i (f)  - \sum_{j=1}^n \nu_i (f) \Big|  \\&
 \le \sum_{i=1}^n |\tau_i(f) - \nu_i(f)| \\&
 <   \sum_{i=1}^n \ep \|\mu_i\| = \ep . 
\end{align*}

To verify the first inequality, first observe that for $a \in E$ and $\xi,\eta\in \Xi$,
\begin{align*}
 \ip{\Phi(a) \xi, \eta} &= \sum_{i=1}^n \sum_{k=1}^{m_i} \tau_{ik}(a) \ip{ \chi_{B_{ik}}\xi, \eta} \\&
 = \sum_{i=1}^n \sum_{k=1}^{m_i} a(x_{ik}) \ip{ \chi_{B_{ik}}\xi, \eta}  
\end{align*}
and
\[
 \ip{\pi_\mu(a) \xi, \eta} = \sum_{i=1}^n \sum_{k=1}^{m_i}  \ip{ a \chi_{B_{ik}}\xi, \eta}.
\]
Recall that $\big| a(x) -a(x_{ik}) \big| < \ep'$ if $x \in K_i$ and $k \in \G_i$ (so that $x_{ik}\in K_i$). Otherwise $\big| a(x) -a(x_{ik}) \big| \le 2 \|a\|_\infty \le 2N$. Also, since $\mu(C)<\delta$, we have $\|\chi_C \xi\|^2 < \ep'/2= \ep/4N$. Thus we obtain
\begin{align*}
 \big| \ip{(\Phi(a) - \pi_\mu(a)) \xi, \eta} \big| &
 \le  \sum_{i=1}^n \sum_{k=1}^{m_i} \big| \bip{\big(a -a(x_{ik}) \big)  \chi_{B_{ik}}\xi, \eta}  \big| \\&
 \le \sum_{i=1}^n \sum_{k\in \G_i} \ep' \| \chi_{B_{ik}}\xi \| \, \| \chi_{B_{ik}} \eta\| \\&\quad
 + \sum_{i=1}^n \sum_{k\not\in \G_i} 2 \|a\|_\infty  \| \chi_{B_{ik}}\xi \| \, \| \chi_{B_{ik}} \eta\| \\&
 \le \frac{\ep}{2N} \left( \sum_{i,k} \| \chi_{B_{ik}}\xi \|^2 \right)^{1/2} \left( \sum_{i,k} \| \chi_{B_{ik}}\eta \|^2 \right)^{1/2} \\&\quad
 + 2N \left(\!\! \sum_{i,k\not\in \G_i}\!\!  \| \chi_{B_{ik}}\xi \|^2 \right)^{\!1/2} \left(\!\! \sum_{i,k\not\in \G_i}\!\! \| \chi_{B_{ik}}\eta \|^2 \right)^{\!1/2}\\&
 < \frac{\ep}{2N} \|\xi\|\,\|\eta\| + 2N \| \chi_C \xi \|\, \| \chi_C \eta\| \\&
 < \frac{\ep}{2} + 2N \frac{\ep}{4N} = \ep .
\end{align*}
This completes the proof of the lemma.
\end{proof}

We are now ready to prove the final piece of the equivalence between the Choquet order and strong dilation order.

\begin{thm}\label{T:choquet-equiv-strong-dilation-order}
Let $K$ be a compact convex subset of a locally convex vector space, and let $\mu,\nu \in M^+(K)$. 
Then $\mu \prec_c \nu$ if and only if $\mu \prec_s \nu$.
\end{thm}

\begin{proof} One direction is provided by Theorem~\ref{T:strong_diln_order_implies_Choquet_order}. We complete the proof of the converse.

First assume that $\mu$ is a continuous measure. Form a net $\Lambda$ consisting of tuples $(E,F,\Xi,\ep)$ where $E \subset A(K)$, $F \subset \rC(K)$ and  $\Xi \subset L^2(\mu)$ are finite sets, and $\ep > 0$. Say $(E,F,\Xi,\ep) \le  (E',F',\Xi',\ep')$ provided that $E \subset E'$, $F\subset F'$, $\Xi\subset \Xi'$ and $\ep > \ep'$. For each $\lambda\in\Lambda$, define a positive map $\Phi_\lambda$ using Lemma~\ref{L:technical lemma}.

The net $\{\Phi_\lambda\}_{\lambda\in\Lambda}$ is bounded by 1 since each $\Phi_\lambda$ is unital and positive. By passing to a cofinal subnet, we obtain a point-\textsc{wot} limit $\Phi$ which will also be a unital positive map from $\rC(K)$ into $L^\infty(\mu)$. By the properties of $\Phi_\lambda$ obtained from Lemma~\ref{L:technical lemma}, we see that for all $a\in A(K)$ and $\xi, \eta \in L^2(\mu)$, 
\[  \ip{\Phi(a) \xi, \eta} = \lim_\Lambda \ip{\Phi_\lambda(a) \xi, \eta} = \ip{\pi_\mu(a) \xi, \eta}. \]
Furthermore, for all $f \in \rC(K)$,
\[ \ip{\Phi(f) 1_\mu, 1_\mu} = \lim_\Lambda \ip{\Phi_\lambda(f) 1_\mu, 1_\mu} = \nu(f)  .\]
It follows that $\Phi$ satisfies the hypotheses of Proposition~\ref{P:cp version}. Therefore $\mu \prec_s \nu$.

For a general measure $\mu$, we decompose $\mu$ as $\mu=\mu_d + \mu_c$ where $\mu_d$ is the discrete (atomic) part of $\mu$ and $\mu_c$ is the continuous part. As indicated before Lemma~\ref{L:technical lemma}, we use the Cartier-Fell-Meyer theorem to decompose $\nu$ as $\nu=\nu_1+\nu_2$ with $\mu_d \prec_c\nu_1$ and $\mu_c\prec_c\nu_2$. Applying Proposition~\ref{choquet-implies-dilation-discrete} to the first pair of measures and the previous paragraph to the second pair, we obtain $\mu_d \prec_s \nu_1$ and $\mu_c \prec_s \nu_2$. The result now follows from Lemma~\ref{lem:dilation-order-additive}.
\end{proof}

As a restatement of the results in this section, we  have the following alternative description of the Choquet order.

\begin{cor}\label{C:cp-version-of-choquet-order}
Let $K$ be a compact convex subset of a locally compact vector space, and let $\mu,\nu \in M^+(K)$. 
Then $\mu \prec_c \nu$ if and only if there is a unital positive map $\Phi:\rC(K) \to L^\infty(\mu)$ such that 
\begin{enumerate}
\item $\Phi(a)=a$ \ \ for all $a \in A(K)$, and
\item $\nu(f) = \int_K \Phi(f) \,d\mu$ \ \ for all $f \in \rC(K)$.
\end{enumerate}
\end{cor}

\section{Extending Cartier's Theorem} \label{S:cartier}

Now we will explain how Cartier's theorem can be extended to the non-metrizable setting. Instead of utilizing the theory of disintegration of measures, we will instead utilize Theorem~\ref{T:choquet-equiv-strong-dilation-order}, as well as a lifting theorem of Maharam \cite{Maharam} (see \cite{TulceaI} or \cite{Tulcea_book} for an alternate proof).

Let $(X, \B,\mu)$ be a complete probability space, and let $M^\infty(\B)$ denote the space of bounded measurable functions on $X$. There is a natural quotient map $q:M^\infty(\B) \to L^\infty(\mu)$ obtained by identifying functions which agree $\mu$-a.e. 
Using C*-algebraic terminology, Maharam's lifting theorem says that there is a unital $*$-monomorphism $\rho: L^\infty(\mu) \to M^\infty(\B)$ such that $q \circ \rho$ is the identity map on $L^\infty(\mu)$. In other words, there is a positive unital lifting of this quotient map which is also multiplicative. We actually only require the positivity of this lifting, and not the fact that it is a homomorphism.

\begin{thm} \label{T:cartier} 
Let $K$ be a compact convex subset of a locally convex vector space. Let $\mu, \nu \in M^+(K)$ satisfy $\mu \prec_c \nu$.
Then there is a family $\{\lambda_x\}_{x \in K} \subset P(K)$ of probability measures such that
\begin{enumerate}
\item $\lambda_x(a) = a(x)$ $\mu$-a.e. for all $a \in A(K)$,
\item $f \to \lambda_x(f)$ is $\mu$-measurable for all $f \in \rC(K)$, and
\item $\int f \,d\nu = \int \lambda_x(f) \,d\mu$ for all $f \in \rC(K)$.
\end{enumerate}
\end{thm}

\begin{proof}
By Theorem~\ref{T:choquet-equiv-strong-dilation-order}, we have $\mu \prec_s \nu$.
Thus we have a unital  positive map $\Phi : \rC(K) \to L^\infty(\mu)$ such that 
$\Phi(a)=a$ for $a \in A(K)$ and $\nu(f) = \int_K \Phi(f) \,d\mu$ for all $f \in \rC(K)$.
Let $\rho: L^\infty(\mu) \to M^\infty(\B)$ be a positive unital lifting to the bounded measurable functions on $K$.
Then for each $x \in K$, the positive linear map $\phi_x : \rC(K) \to \bC$ defined by
\[  \phi_x(f) = \rho \circ \Phi(f) (x) \qfor f \in \rC(K) \]
satisfies $\phi_x(1)= \rho \circ \Phi(1)(x) = 1$. So this is a state on $\rC(K)$. By the Riesz-Markov-Kakutani representation theorem, there is a regular Borel probability measure $\lambda_x \in P(K)$ so that $\phi_x(f) = \int f \,d\lambda_x$. Therefore $f \to \lambda_x(f) = \rho \circ \Phi(f) (x)$ is measurable for every $f\in\rC(K)$, and
\[ \int f \,d\nu = \int \Phi(f) \,d\mu = \int \phi_x(f) \,d\mu(x) = \int \lambda_x(f) \, d\mu(x) .\]
Finally for $a \in A(K)$, since $\Phi(a) = a$, we have $\rho(a) = a$ $\mu$-a.e., and therefore $\lambda_x(a) = a(x)$ $\mu$-a.e.
\end{proof}

\begin{rem}
Theorem \ref{T:cartier} is not quite ideal, because in general we do not know that $\lambda_x \in P_x(K)$ for $\mu$-a.e.\ $x \in K$. 
So it seems that our result has a somewhat weaker conclusion than Cartier's theorem. However, it turns out that our result is equivalent to Cartier's theorem in the metrizable case, as we now explain.

The issue is that we do not know whether the lifting map $\rho$ in our proof can always be chosen to satisfy $\rho(a)=a$ for all $a\in\rC(K)$. Such a lifting is called a strong lifting. If $\rho$ is a strong lifting, then $\phi_x(a) = \rho \circ \Phi(a)(x) = a(x)$ for all $a \in A(K)$, and therefore $\lambda_x \in P_x(K)$.

While strong liftings exist in the metrizable setting and for certain product measures, in general the existence of a strong lifting is a major open problem (see e.g.\ \cite{Tulcea_book}). 
\end{rem}

We now recover Cartier's theorem.

\begin{cor}[Cartier] \label{C:cartier}
Let $K$ be a metrizable compact convex subset of a locally convex vector space. Let $\mu, \nu \in M^+(K)$ satisfy $\mu \prec_c \nu$.
Then there is a family $\{\lambda_x\}_{x \in K} \subset P(K)$ of probability measures such that
\begin{enumerate}
\item  $\lambda_x \in P_x(K)$ for all $x \in K$,
\item $f \to \lambda_x(f)$ is $\mu$-measurable for all $f \in \rC(K)$, and
\item $\int f \,d\nu = \int \lambda_x(f) \,d\mu$ for all $f \in \rC(K)$.
\end{enumerate} 
\end{cor}

\begin{proof}
Let $\{\lambda_x\}_{x \in K} \subset P(K)$ be probability measures that satisfy Theorem~\ref{T:cartier}.
Since  $K$ is metrizable, $A(K)$ is separable. Let $\{ a_n : n\ge1 \}$ be a countable spanning set.
For each $n\ge1$, there is a set $E_n$ with $\mu(E_n)=0$ so that $\lambda_x(a_n) = a_n(x)$ for all $x \in K\setminus E_n$.
Let $E = \bigcup_{n\ge1} E_n$.
Then by density, we have that $\lambda_x(a) = a(x)$ for all $a \in A(K)$ and all $x \in K\setminus E$.
Thus $\lambda_x \in P_x(K)$ if $x \not\in E$.
Since $\mu(E)=0$, we may redefine $\lambda_x := \delta_x$ for $x \in E$.
This has no effect on (2) or (3), but now ensures that (1) also holds.
\end{proof}

\section{Extending \u Sa\u skin's Theorem} \label{S:saskin}

A classical result of Korovkin in approximation theorem states that if $\phi_n: \rC[0,1] \to \rC[0,1]$ is a sequence of positive maps satisfying 
$\lim_n \|\phi_n(g) - g\| = 0$ for each $g \in \{1,x,x^2\}$, then $\lim_n \|\phi_n(f) - f\| = 0$ for all $f \in \rC[0,1]$.

\u Sa\u skin \cite{Sas1967} proved a much more general version of Korovkin's theorem in the setting of a commutative C*-algebra $\rC(X)$ of continuous functions on a compact metrizable space $X$. A subset $G \subset \rC(X)$ containing the constant function $1$ is said to be a {\em Korovkin set} if whenever $\phi_n : \rC(X) \to \rC(X)$ is a sequence of positive linear maps satisfying $\lim_n \|\phi_n(g) - g\| = 0$ for all $g \in G$, then $\lim_n \|\phi_n(f) - f\| = 0$ for all $f \in \rC(X)$. 
\u Sa\u skin proved that if $X$ is metrizable, then a subset $G \subset \rC(X)$ that separates points and contains $1$ is a Korovkin set if and only if $\partial_A X = X$, where $A = \overline{\operatorname{span}(G \cup G^*)}$ denotes the function system generated by $G$.

We will show that \u Sa\u skin's theorem extends to the non-metrizable setting if sequences are replaced by nets.
Our proof will utilize ideas from the theory of operator algebras, and will provide a new proof of \u Sa\u skin's theorem in the metrizable case ((where sequences suffice).

We are motivated by work of Arveson \cite{Arv2011} on non-commu\-ta\-tive versions of these results.
His results will be addressed later in this paper. In this section we will restrict ourselves to the classical commutative setting.

\begin{thm}\label{T:weak-uep}
Let $A$ be a function system in $\rC(X)$ for a compact Hausdorff space $X$. Then the following are equivalent:
\begin{enumerate}
\item $\partial_A X = X$.
\item For any compact Hausdorff space $Y$, unital $*$-homomorphism $\pi:\rC(X) \to \rC(Y)$ and positive linear map $\phi:\rC(X) \to \rC(Y)$ satisfying $\phi|_A = \pi|_A$, it follows that $\phi=\pi$.
\end{enumerate}
\end{thm}

\begin{proof}
Suppose first that (1) fails. That is, there is some $x_0\in X$ and regular probability measure $\mu \ne \delta_{x_0}$ on $X$ which represents $x_0$. Define $\pi:\rC(X) \to \bC$ by $\pi(f) = f(x_0)$ and $\phi(f) = \int_X f \, d\mu$. Then clearly (2) fails.

Conversely, suppose that (1) holds, and let $\pi$ and $\phi$ be given as in (2). We first provide the easy proof of a classical fact.
Notice that for each $y\in Y$, the linear functional map $\psi_y$ on $\rC(X)$ given by $\psi_y( f) = \pi(f)(y)$ is multiplicative.
Since the maximal ideal space of $\rC(X)$ is $X$, there is a unique $x := \pi^*(y) \in X$ such that $\psi_y(f) = f(x)$ for all $f\in\rC(X)$.
It is easy to verify that $\pi^*$ is continuous. Hence $\pi(f) = f \circ \pi^*$.

Fix $y\in Y$, and define the state $\phi_y(f) = \phi(f)(y)$. 
By the Riesz-Markov theorem, there is a regular probability measure $\mu_y$ so that $\phi_y(f) = \int_X f \,d\mu_y$. 
By assumption, $\phi_y(a) = \pi(a)(y) = a(\pi^* y)$ for all $a\in A$. That is, $\mu_y$ is a representing measure for $\pi^* y$. 
The assumption that $\partial_A X = X$ means that every $x\in X$ has a unique representing measure; whence $\mu_y = \delta_{\pi^* y}$.
Therefore $\phi_y(f) = f(\pi^* y) = \pi(f)(y)$ for all $f\in\rC(X)$. Hence $\phi=\pi$ and (2) holds.
\end{proof}

\begin{defn}
Let $X$ be a compact Hausdorff space.
A subset $G \subset \rC(X)$ containing $1$ is a {\em Korovkin set} if whenever $\phi_\lambda:\rC(X) \to \rC(X)$ for $\lambda\in\Lambda$ is a net of positive maps satisfying $\lim_\lambda \|\phi_\lambda(g) - g\| = 0$ for all $g \in G$, then $\lim_\lambda \|\phi_\lambda(f) - f\| = 0$ for all $f \in \rC(X)$. If $X$ is metrizable, nets can be replaced by sequences.
\end{defn}

The following result is our extension of \u Sa\u skin's theorem.

\begin{thm}\label{T:saskin}
Let $X$ be a compact Hausdorff space, and let $G$ be a subset of $\rC(X)$ containing the constant function $1$.
Let $A$ be the function system in $\rC(X)$ spanned by $G \cup G^*$. Then the following are equivalent:
\begin{enumerate}
\item $\partial_A X = X$.
\item $G$ is a Korovkin set.
\item For any compact Hausdorff space $Y$, unital $*$-homomorphism $\pi:\rC(X) \to \rC(Y)$ and net of positive linear maps $\phi_\lambda:\rC(X) \to \rC(Y)$ satisfying $\lim_\lambda \phi_\lambda(g) = \pi(g)$ for all $g \in G$, it follows that $\lim_\lambda \phi_\lambda(f) = \pi(f)$ for all $f \in \rC(X)$.
\end{enumerate}
\end{thm}

\begin{proof}
Replacing $Y$ by $X$ in (3) yields (2). If (2) holds, let $x \in X$ and let $\mu$ be a representing measure for $x$ for the function system $A$.
Fix a neighbourhood base $\O$ of open sets containing $x$. In the metrizable case, this can be chosen to be $\O=\{ b_{1/n}(x) : n\ge1\}$.
Make $\O$ into a net by setting $U < V$ is $U \supset V$. The fact that it is a base ensures that $\O$ is directed in this partial order.
By Urysohn's Lemma and the axiom of choice, there are continuous functions $h_U\in \rC(X)$ such that $h_U(x)=1$, $0 \le h_u \le 1$ and $\supp(h_U) \subset U$.
Define positive maps $\phi_U : \rC(X) \to \rC(X)$ by
\[ \phi_U(f) = f(1-h_U) + \mu(f) h_U \qfor f \in \rC(X) .\]
For any $g\in G$ and $\ep>0$, there is a neighbourhood $U \in \O$ such that $|g(x)-g(y)| < \ep$ for all $y\in U$.
Thus if $V \ge U$,
\[ \| g - \phi_V(g) \|  = \| (g - \mu(g)) h_V \| \le \sup_{y\in V} |g(y) - g(x) | < \ep .\]
So $\phi_V$ converges to the identity on $G$. Since $\phi_V$ are all positive, they also converge on $G^*$ and hence on the span of $G\cup G^*$.
Moreover $\|\phi_V\| = \|\phi_V(1)\| < 1 + \ep$, so they are equicontinuous and this extends to the closed span.
By property (2), this sequence must converge for all functions $f\in\rC(X)$.
In particular, evaluating at $x$, we see that
\[ f(x) = \lim_U \phi_U(f)(x) = \mu(f) .\]
So $x$ has a unique representing measure, and therefore $x \in \partial_A X$. That is, (1) holds.

If (1) holds, let $\pi:\rC(X) \to \rC(Y)$ be a $*$-homomorphism and consider any net of positive linear maps $\phi_\lambda:\rC(X) \to \rC(Y)$ satisfying $\lim_\lambda \phi_\lambda(g) = \pi(g)$ for all $g \in G$. Since $1 \in G$ and thus $\lim_\lambda \phi_\lambda(1)=1$, we may suppose that $\| \phi_\lambda(1) - 1 \| \le .5$ for $\lambda \ge \lambda_0$ and restrict the net $\Lambda$ to this cofinal subset. Then we have $\|\phi_\lambda\| = \|\phi_\lambda(1)\| \le 1.5$.

Let $l^\infty(\Lambda, \rC(X))$ denote the abelian C*-algebra of bounded functions from $\Lambda$ into $\rC(X)$.
This algebra has an ideal $\J = c_0(\Lambda, \rC(X))$ consisting of all functions $f=(f_\lambda)$ with $\lim_\lambda \|f_\lambda\| = 0$.
The C*-algebra $B = l^\infty(\Lambda, \rC(X)/\J$ is abelian, and thus by Gelfand theory is isomorphic to $\rC(Y)$ for some compact Hausdorff space $Y$.
Define a $*$-homomorphism $\pi:\rC(X) \to \rC(Y)$ by $\pi(f) =  (f,f,\dots) + \J$; i.e. the function is the constant function $f_\lambda=f$ modulo the ideal $\J$. Also define a positive map $\phi: \rC(X) \to \rC(Y)$ by $\phi(f) = (\phi_\lambda(f)) + \J$. The hypotheses guarantee that $\phi|_A = \pi|_A$.
Therefore by Theorem~\ref{T:weak-uep}, $\phi=\pi$. This means that $(f - \phi_\lambda(f))$ belongs to $\J$; that is, $\lim_\lambda \phi_\lambda(f) = \pi(f)$ for all $f \in \rC(X)$. So (3) holds.
\end{proof}

Kleski \cite{Kle2014}  has obtained a version of Korovkin's Theorem that applies when the C*-algebra generated by the operator system is separable and of type I, 
and the range of the maps belongs to the double commutant. 
This includes function systems in $C(X)$ for $X$ metrizable and unital completely positive maps with range in $L^\infty(X,\mu)$ for a probability measure $\mu$ on $X$.

\section{Dilation order} \label{sec:dilation-order}

We now turn to the relationship between the classical commutative setting and the noncommutative setting. In this section we will introduce the dilation order for measures on a compact convex set. We will see that this notion is closely related to Arveson's hyperrigidity conjecture. We will also establish its relationship to the Choquet order.

Recall Definition~\ref{D:representation} of a $*$-representation of a measure.

\begin{defn} \label{defn:dilation}
Let $K$ be a compact convex subset of a locally convex vector space and let $\mu,\nu \in M^+(K)$ be measures. We say that a repre\-sen\-ta\-tion $(\pi, H, \xi)$ of $\mu$ is {\em dilated} by a repre\-sen\-ta\-tion $(\sigma, L, \eta)$ of $\nu$ if there is an isometry $V : H \to L$ such that
\begin{enumerate}
\item $\pi(a) = V^* \sigma(a) V$ \ for every $a \in A(K)$, and
\item $\eta = V\xi$.
\end{enumerate}
\end{defn}

\begin{defn} \label{defn:dilation-order}
Let $K$ be a compact convex subset of a locally convex vector space. The {\em dilation order} ``$\prec_d$'' on $M^+(K)$ is defined for $\mu,\nu \in M^+(K)$ by $\mu \prec_d \nu$ if some repre\-sen\-ta\-tion of $\mu$ is dilated by a repre\-sen\-ta\-tion of $\nu$ in the sense of Definition~\ref{defn:dilation}.
\end{defn}

The next two propositions are useful for working with the dilation order.

\begin{prop} \label{prop:non-symmetric-dilation}
Let $K$ be a compact convex subset of a locally convex vector space and let $\mu,\nu \in M^+(K)$. If $\mu \prec_d \nu$, then the GNS repre\-sen\-ta\-tion of $\mu$ is dilated by a repre\-sen\-ta\-tion of $\nu$.
\end{prop}

\begin{proof}
Suppose that $\mu \prec_d \nu$. Then some repre\-sen\-ta\-tion $(\pi, H, \xi)$ of $\mu$ is dilated by a repre\-sen\-ta\-tion $(\sigma, L, \eta)$ of $\nu$. We claim that the GNS repre\-sen\-ta\-tion $(\pi_\mu, L^2(\mu), 1_\mu)$ of $\mu$ is also dilated by $(\sigma, L, \eta)$. To see this, let $V : H \to L$ be the isometry as in Definition \ref{defn:dilation}. By the remarks in Section \ref{sec:comm-c-star-alg}, the restriction of $\pi$ to the cyclic invariant subspace for $\pi$ generated by $\xi$ is unitarily equivalent to $\pi_\mu$ via a unitary that maps $1_\mu$ to $\xi$. Hence we can assume that $H = L^2(\mu) \oplus H'$ for some Hilbert space $H'$, $\xi = 1_\mu$ and $\pi = \pi_\mu \oplus \pi'$ for some $*$-repre\-sen\-ta\-tion $\pi' : \rC(K) \to \B(H')$. Let $W : L^2(\mu) \to L$ be the restriction $W = V|_{L^2(\mu)}$. Then $W$ is an isometry, $\pi_\mu(a) = W^* \sigma(a) W$ and $\eta = W1_\mu$. Hence $(\pi_\mu, L^2(\mu), 1_\mu)$ is dilated by $(\sigma, L, \eta)$.
\end{proof}

The following characterization of the dilation order looks similar to the definition of the strong dilation order in Definition \ref{defn:strong-dilation-order}. Note that here the range is $\B(H)$ instead of $L^\infty(\mu)$. We will see that this is an important distinction.

\begin{prop} \label{P:cp version}
Let $K$ be a compact convex subset of a locally convex vector space and let $\mu,\nu \in M^+(K)$.
Then $\mu \prec_d \nu$ if and only if there is a positive map $\Phi:\rC(K) \to \B(L^2(\mu))$ such that 
\[ \Phi(a) = \pi_\mu(a) \qforal a \in A(K) ,\]
and
\[ \nu(f) = \ip{\Phi(f) 1_\mu, 1_\mu} \qforal f \in \rC(K).\]
\end{prop}

\begin{proof}
First suppose that $\mu \prec_d \nu$. By Proposition~\ref{prop:non-symmetric-dilation}, there is an isometry $V:L^2(\mu) \to L$ and a representation $(\sigma,  L,\eta)$ such that $V1_\mu = \eta$ and $\pi_\mu(a) = V^*\sigma(a)V$ for all $a\in A(K)$.
The positive map $\Phi$ into $\B(L^2(\mu))$ is given by $\Phi(f) = V^*\sigma(f)V$. Then for all $f \in \rC(K)$,
\[ \nu(f) = \ip{\sigma(f) \eta, \eta} =  \ip{\sigma(f) V 1_\mu, V 1_\mu} = \ip{\Phi(f) 1_\mu, 1_\mu} .\]

For the converse, note that since $\rC(K)$ is commutative, a result of Stinespring \cite{St1955} (see \cite{Paulsen}*{Theorem 3.11}) implies that $\Phi$ is completely positive. Therefore, we may apply Stinespring's dilation theorem \cite{St1955} (see \cite{Paulsen}*{Theorem 4.1}) to obtain a $*$-repre\-sen\-ta\-tion $\sigma:\rC(K) \to \B(L)$ and an isometry $V:L^2(\mu) \to L$ such that
\[ \Phi(f) = V^* \sigma(f) V \qforal f \in \rC(K) .\]
In particular,
\[ \pi_\mu(a) = V^* \sigma(a) V \qforal a \in A(K), \]
and setting $\eta = V1_\mu$, it follows that for all $f \in \rC(K)$,
\[ \ip{ \sigma(f) \eta, \eta} = \ip{ V^* \sigma(f) V 1_\mu, 1_\mu} = \ip{\Phi(f) 1_\mu, 1_\mu} = \nu(f) .\]
Therefore $\mu \prec_d \nu$.
\end{proof}

\begin{cor}\label{C:diln-order-implies-choquet}
Let $K$ be a compact convex subset of a locally convex vector space and let $\mu,\nu \in M^+(K)$.
Then $\mu \prec_c \nu$ implies that $\mu \prec_d \nu$.
\end{cor}

\begin{proof}
It is evident that the strong dilation order is more restrictive than the dilation order.
So this result is immediate from Theorem~\ref{T:choquet-equiv-strong-dilation-order}.
\end{proof}

A continuous function defined on an interval $J \subset \bR$ is called \textit{operator convex} provided that
whenever $X, Y$ are self-adjoint operators with spectra contained in $J$, we have
\[ f(tX + (1-t)Y) \le t f(A) + (1-t) f(Y) .\]
These functions are related to the operator monotone functions of L\"owner, and were studied by Bendat and Sherman \cite{BS1955}.
They show that like operator monotone functions, these operator convex functions extend to be analytic in the upper half plane.
A key role is played by the functions $f_t(x) = x^2(1-tx)^{-1}$, which are operator convex on $(-1,1)$ provided that $|t|<1$.
Hansen and Pedersen \cites{HP1982, HP2003} introduced the notion of truly operator convex combinations, and showed in the 
latter paper that operator convex functions satisfy even stronger inequalities, where the scalar $t$ is replaced by a contraction. Specifically, they showed that when $X_1,\dots, X_n$ are self-adjoint operators
with spectra in $J$ and $A_i$ are operators satisfying $\sum_{i=1}^n A_i^*A_i=I$, then
\[ f \big( \sum_{i=1}^n A_i^* X_iA_i \big) \le \sum_{i=1}^n A_i^* f(X_i) A_i .\]

Define $\OP(K)$ to be the closed cone of functions generated by
\[ \{ f(a) : a=a^* \in A(K),\ f \text{ operator convex on interval } J \supset \sigma(a) \} .\]
Then $\OP(K)$ contains the continuous affine functions and is contained in the set $P(K)$ of all convex functions. If $K$ is non-trivial, then this inclusion is proper. 
We will see that
\[ \OP(K) \cap -\OP(K) = A(K) .\]

The critical observation for us is that domination in dilation order implies domination on $\OP(K)$.

\begin{lem}\label{L:diln-order-implies-oc-order}
Let $K$ be a compact convex subset of a locally convex vector space, and let $\mu,\nu \in M^+(K)$. 
If $\mu \prec_d \nu$, then $\mu(f) \le  \nu(f)$ for all $f \in \OP(K)$.
\end{lem}

\begin{proof}
Since $\mu \prec_d \nu$, we have an isometry $V:L^2(\mu) \to L$ and a representation $(\sigma, L, \eta)$ of $\rC(K)$ on $L$
such that $V1_\mu = \eta$, $\nu(f) = \ip{\sigma(f) \eta, \eta}$ and $\pi_\mu(a) = V^*\sigma(a)V$ for $a \in A(K)$.
We let $\Phi(g) = V^* \sigma(g)V$ be the positive map of Proposition~\ref{P:cp version}.

Let $a \in A(K)$ be self-adjoint and let $f$ be an operator convex function on an interval $J \supset \sigma(a)$.
Then by operator convexity,
\begin{align*}
\pi_\mu(f(a)) &= f( \pi_\mu(a) ) \\&
= f( V^*\sigma(a)V ) \le V^* f(\sigma(a))V \\&
= V^* \sigma(f(a)) V = \Phi(f(a)) .
\end{align*}
Consequently, 
\[ \mu(f(a)) = \ip{\pi_\mu(f(a)) 1_\mu, 1_\mu} \le \ip{ \Phi(f(a)) 1_\mu, 1_\mu} = \nu(f(a)) .\]
This extends to positive linear combinations and uniform limits, and thus to all of $\OP(K)$.
\end{proof}

Lemma~\ref{L:diln-order-implies-oc-order} allows us to show that the dilation order is indeed a partial order, and not just a preorder.

\begin{prop} \label{P:oc_order}
Let $K$ be a compact convex subset of a locally convex vector space, and let $\mu,\nu \in M^+(K)$. 
If $\mu$ and $\nu$ coincide on the set $\OP(K)$, then $\mu=\nu$. Thus the dilation order is a partial order.
Moreover $\OP(K) - \OP(K)$ is norm dense in $\rC_\bR(K)$.
\end{prop}

\begin{proof}
Let $a=a^* \in A(K)$ with $\|a\| < 1$. Then $\OP(K)$ contains $f_t(a) = a^2(1-ta)^{-1}$ for $|t|<1$. Note that
\[ \mu(f_t(a)) = \mu\big( \sum_{k\ge0} a^{k+2} t^k \big) = \sum_{k\ge0} \mu(a^{k+2}) t^k .\]
Since $\nu(f_t) = \mu(f_t)$, the coincidence of these two power series on $(-1,1)$ implies that $\nu(a^k) = \mu(a^k)$ for $k \ge 2$.
Since $\mu$ and $\nu$ also agree on $A(K)$, this identity is valid for $k\ge0$.

Now let $a_1,\dots, a_n$ be elements of $A(K)$ with $\|a_i\| \le1$.
Then $\sum_{i=1}^n t_i a_i$ are contractions in $A(K)$ for $|t_i|<\frac1n$.
Arguing as above, we obtain that for $k \ge 0$, 
\[
  \mu \big( \big( \sum_{i=1}^n t_i a_i \big)^k \big) 
  = \sum_{\substack{k_i\ge 0\\ k_1 + \dots + k_n = k}} \frac{k!}{k_1! \dots k_n!} \mu(a_1^{k_1}\dots a_n^{k_n}) t_1^{k_1}\dots t_n^{k_n}
\]
agrees with $\nu \big( \big( \sum_{i=1}^n t_i a_i \big)^k \big)$ on $(-\frac1n,\frac1n)^n$. Again this forces a coincidence of the Taylor coefficients.
Consequently
\[ \nu (a_1^{k_1}\dots a_n^{k_n}) = \mu (a_1^{k_1}\dots a_n^{k_n}) \qforal a_i\in A(K) \AND k_i \ge 0 .\]
It follows that $\mu$ and $\nu$ agree on the algebra generated by $A(K)$, namely $\rC(K)$.
Thus $\mu=\nu$. In particular, if $\mu \prec_d \nu \prec_d \mu$, then by Lemma~\ref{L:diln-order-implies-oc-order}, $\mu$ and $\nu$ agree on $\OP(K)$, and therefore $\mu=\nu$.

The Hahn-Banach Theorem now implies that the real span of $\OP(K)$, which is $\OP(K) - \OP(K)$, is dense in $\rC_\bR(K)$.
\end{proof}

\section{The unique extension property} \label{sec:unique-extension-property}

In this section we will relate the dilation order to some constructions in the theory of operator algebras.

While most of the results from previous sections in this paper can be stated in terms of classical commutative notions, for what follows we will require the theory of completely positive maps and noncommutative C*-algebras.

Recall that if $A$ is a function system, then a map $\phi: A \to \B(H)$ is {\em completely positive} if
each of the maps $\phi^{(n)}: \M_n(A) \to \M_n(\B(H))$ defined by
\[ \phi^{(n)}\big(\big[ a_{ij} \big]\big) = \big[ \phi(a_{ij}) \big] \]
are positive for all $n \geq 1$. For the basic theory of completely positive maps, we refer the reader to Paulsen's book \cite{Paulsen}.

Stinespring proved (see e.g.\ \cite{Paulsen}*{Theorem 3.11}) that if $A$ is a commutative C*-algebra, then every positive map $\phi: A \to \B(H)$ is automatically completely positive. However, this is not true for an arbitrary function system (see e.g.\ \cite{Paulsen}*{Example 2.2}).

It is an important fact that $\B(H)$ is injective in the category of operator systems and completely positive maps  \cite{Arv1969} (see \cite{Paulsen}*{Theorem 7.5}). It follows from this that every completely positive map $\phi:A \to \B(H)$ has a completely positive extension to $C := \ca(A)$. On the other hand, if $\phi$ is positive but not completely positive, then it does not extend to a positive map on $C$, because such an extension would be completely positive, and would therefore imply the complete positivity of $\phi$.

\begin{defn}
Let $A$ be a concrete function system that generates a commutative C*-algebra $\rC(X)$. A unital  completely positive map $\phi:A \to \B(H)$  is said to have the {\em unique extension property} if it has a unique completely positive extension $\pi:C \to \B(H)$ and $\pi$ is a $*$-homomorphism.
Similarly, we will say that a $*$-repre\-sen\-ta\-tion $\pi : C \to \B(H)$ has the {\em unique extension property relative to $A$} if the restriction $\phi = \pi|_A$ has the unique extension property.
\end{defn}

\begin{defn} \label{defn:maximal-map}
Let $A$ be a function system and let $\phi : A \to \B(H)$ and $\psi : A \to \B(L)$ be unital completely positive maps. We say that $\phi$ is {\em dilated} by $\psi$ or that $\psi$ is a {\em dilation} of $\phi$, and write $\phi \prec \psi$, if there is an isometry $V : H \to L$ such that $\phi(a) = V^* \psi(a)V$ for all $a \in A$. If, in addition, the subspace $VH$ is invariant (and hence reducing) for $\psi(A)$, then $\psi$ is said to be a {\em trivial dilation} of $\phi$. If every dilation of $\phi$ is trivial, then $\phi$ is said to be {\em maximal}.
\end{defn}

The above notions naturally extend to the more general setting of operator systems, which are unital self-adjoint subspaces of  possibly noncommutative C*-algebras. In this general setting, completely positive maps are essential, since positive maps, even on C*-algebras, need not be completely positive. 

In his seminal work on noncommutative dilation theory, Arveson \cite{Arv1969} outlined many of the key ideas for a theory of dilations of operator algebras and operator systems. A central theme in his work is the notion of a boundary repre\-sen\-ta\-tion.

\begin{defn}[Arveson] \label{defn:boundary-representation}
Let $A$ be a concrete operator system. An irreducible representation of $\ca(A)$ is  a {\em boundary representation} of $A$ if it has the unique extension property relative to $A$. The {\em noncommutative Choquet boundary} is the set of all boundary representations of $A$.
\end{defn}

In the case of function systems, $\ca(A) = \rC(X)$ and the irreducible representations are the point evaluations $\ep_x$ at points of $X$.
The unique extension property in this case is just the statement that $x$ has a unique representing measure; that is, $x \in \partial_A X$.
So Arveson's notion extends the classical notion of the Choquet boundary.

Arveson conjectured that every concrete operator system has a noncommutative Choquet boundary with the property that the boundary representations provide a completely isometric representation of the operator system. Arveson's conjecture was established in the separable case by Arveson himself \cite{Arv2008}, and in the general case by the authors \cite{DK2015}.

The next result was established for (generally non-self-adjoint) operator algebras by Dritschel and McCullough \cite{DritMcCull}, and for operator systems by Arveson \cite{Arv2008}, both in the non-commutative setting.

\begin{prop}[Dritschel-McCullough, Arveson] \label{prop:unique-extn-iff-maximal}
Let $A$ be a concrete function system that generates a commutative C*-algebra $\rC(X)$. A $*$-repre\-sen\-ta\-tion $\pi : \rC(X) \to \B(H)$ has the unique extension property relative to $A$ if and only if the restriction $\pi|_A$ is maximal.
\end{prop}

They also establish that maximal dilations always exist, again in the possibly non-commutative setting \cites{DritMcCull, Arv2008}.

\begin{prop}[Dritschel-McCullough, Arveson] \label{P:maximal-dilations-exist}
Let $A$ be a concrete function system that generates a commutative C*-algebra $\rC(X)$, and let $\phi:A \to \B(H)$ be a completely positive unital map.
Then there is a maximal completely positive map $\psi:A \to \B(L)$ which dilates $\phi$.
\end{prop}

We now make a crucial connection between measures and the unique extension property.

\begin{thm} \label{thm:uep-iff-boundary-measure}
Let $K$ be a compact convex subset of a locally convex vector space, and let $\mu \in M^+(K)$ be a measure with corresponding GNS repre\-sen\-ta\-tion $\pi_\mu : \rC(K) \to \B(L^2(\mu))$. Then $\mu$ is maximal in the dilation order if and only if $\pi_\mu$ has the unique extension property relative to $A(K)$.
\end{thm}

\begin{proof}
Suppose that $\mu$ is maximal in the dilation order. We must show that if $\psi : \rC(K) \to \B(L^2(\mu))$ is a (completely) positive extension of the restriction $\pi_\mu|_{A(K)}$, then $\psi = \pi_\mu$.

By Proposition \ref{prop:unique-extn-iff-maximal}, it suffices to show that $\pi_\mu|_{A(K)}$ is maximal in the sense of Definition \ref{defn:maximal-map}. 
By Proposition~\ref{P:maximal-dilations-exist}, there is a maximal unital completely positive map $\phi : A(K) \to \B(L)$ on a Hilbert space $L$ that dilates $\pi_\mu|_{A(K)}$. 
Thus there is an isometry $V : H \to L$ such that $\pi_\mu(a) = V^* \phi(a) V$ for all $a \in A(K)$. 
By Proposition \ref{prop:unique-extn-iff-maximal}, $\phi$ extends to a $*$-repre\-sen\-ta\-tion $\sigma : \rC(K) \to \B(L)$.

Let $\xi = V1_\mu$, and define $\nu \in P(K)$ by
\[
\nu(f) = \ip{\sigma(f) \xi, \xi} \qfor f \in \rC(K).
\]
Then $(\sigma, L, \xi)$ is a repre\-sen\-ta\-tion of $\nu$. Moreover,
\[
\pi_\mu(a) = V^*\phi(a)V = V^*\sigma(a)V \qforal a \in A(K).
\]
Hence the GNS repre\-sen\-ta\-tion $(\pi_\mu, L^2(\mu), 1_\mu)$ is dilated by $(\sigma, L, \xi)$, implying $\mu \prec_d \nu$. By the maximality of $\mu$ it follows that $\mu = \nu$.

By the remarks in Section \ref{sec:comm-c-star-alg}, the restriction of $\sigma$ to the cyclic invariant subspace generated by $\xi$ is unitarily equivalent to $\pi_\mu$. In other words, the restriction $\pi_\mu|_{A(K)}$ is unitarily equivalent to a summand of $\phi$. Since $\phi$ is maximal, it follows that $\pi_\mu|_{A(K)}$ is necessarily maximal.

Conversely, suppose that the restriction $\pi_\mu|_{A(K)}$ has the unique extension property. Then by Proposition \ref{prop:unique-extn-iff-maximal}, $\pi_\mu|_{A(K)}$ is maximal in the sense of Definition \ref{defn:maximal-map}.

Let $\nu \in P(K)$ be a probability measure such that $\mu \prec_d \nu$. Then by Proposition \ref{prop:non-symmetric-dilation}, there is a repre\-sen\-ta\-tion $(\sigma, L, \xi)$ of $\nu$ that dilates the GNS repre\-sen\-ta\-tion $(\pi_\mu, L^2(\mu), 1_\mu)$. Let $V : L^2(\mu) \to L$ be an isometry implementing this dilation, i.e. such that $\pi_\mu(a) = V^* \sigma(a) V$ for every $a \in A(K)$ and $V1_\mu = \xi$.

The restriction $\pi_\mu|_{A(K)}$ is dilated by the restriction $\sigma|_{A(K)}$. Hence, by the maximality of $\pi_\mu|_{A(K)}$, the subspace $VH$ is invariant for $\sigma$. It follows that the restriction of $\sigma$ to $VH$ is unitarily equivalent to $\pi_\mu$ via a unitary that maps $1_\mu$ to $\xi$. Therefore, $\mu = \nu$ and $\mu$ is maximal in the dilation order. 
\end{proof}

\begin{cor}
Let $K$ be a compact convex subset of a locally convex vector space, and let $\mu \in M^+(K)$.
Then there is $\nu\in M^+(K)$ with $\mu \prec_d \nu$ which is maximal in the dilation order.
\end{cor}

\begin{proof}
Since $\pi_\mu$ is a $*$-representation of $\rC(K)$, the restriction $\pi_\mu|_{A(K)}$ is completely positive.
By Proposition~\ref{P:maximal-dilations-exist}, it has a maximal dilation $\psi$.
So by Proposition~\ref{prop:unique-extn-iff-maximal}, $\psi = \pi_\nu|_{A(K)}$ is the restriction of a $*$-representation 
with the unique extension property with respect to $A(K)$. In particular, $\mu \prec_d \nu$.
By Theorem~\ref{thm:uep-iff-boundary-measure}, $\nu$ is maximal in the dilation order.
\end{proof}

Let $A$ be a concrete function system that generates a commutative C*-algebra $\rC(X)$. Theorem \ref{thm:function-systems} implies that there is an order isomorphism $\iota : A \to A(K)$, where $A(K)$ denotes the function system of continuous affine functions on the state space $K = S(A)$ of $A$. Furthermore, there is a $*$-homomorphism $q : \rC(K) \to \rC(X)$ such that $\iota \circ q$ is the identity on $A$. Consider the adjoint $q : \rC(X)^* \to \rC(K)^*$. As in Section \ref{sec:choquet-boundary}, identifying $X$ and $K$ with the corresponding point evaluations, $q^*$ maps $X$ into $K$ and $q^*(\partial_A X) = \partial K$.
 
More generally, every $*$-representation of $\rC(X)$ gives rise to a $*$-representation of $\rC(K)$ via composition with $q$, and Theorem \ref{thm:uep-iff-boundary-measure} completely characterizes representations of $\rC(K)$ with the unique extension property relative to $A(K)$. While it is not immediately clear that this can be used to characterize $*$-representations of $\rC(X)$ with the unique extension property relative to $A$, the next result shows that this is indeed the case.

From the operator algebraic viewpoint, the result reduces to the fact that the maximality of a unital completely positive map is an intrinsic property of the operator system, and does not depend on the C*-algebra that it generates.

\begin{thm} \label{thm:translation-uep}
Let $A$ be a concrete function system that generates a commutative C*-algebra $\rC(X)$. Let $K = S(A)$ denote the state space of $A$, let $\iota : A \to A(K)$ denote the canonical order isomorphism, and let $q : \rC(K) \to \rC(X)$ denote the canonical quotient map as in Theorem $\ref{thm:function-systems}$. A $*$-representation $\pi : \rC(X) \to \B(H)$ has the unique extension property relative to $A$ if and only if the $*$-representation $\pi \circ q : \rC(K) \to \B(H)$ has the unique extension property relative to $A(K)$.
\end{thm}

\begin{proof}
It is clear that if $\pi \circ q$ has the unique extension property relative to $A(K)$, then $\pi$ has the unique extension property relative to $A$. For the converse, suppose that $\pi \circ q$ does not have the unique extension property relative to $A(K)$. Then by Proposition \ref{prop:unique-extn-iff-maximal}, the restriction $\pi \circ q|_{A(K)}$ is not maximal. Thus there is unital completely positive map $\phi : A(K) \to B(L)$ and an isometry $V : H \to L$ such that $VH$ is not invariant for $\phi(A(K))$ and $\pi \circ q (b) = V^*\phi (b) V$ for all $a \in A(K)$.

The order isomorphism $\iota$ satisfies $q \circ \iota(a) = a$ for all $a \in A$. Thus, defining $\psi : A \to B(L)$ by $\psi(a) = \phi \circ \iota(a)$ for $a \in A$, it follows from above that $VH$ is not invariant for $\psi(A)$ and $\pi(a) = V^* \psi(a) V$ for all $a \in A$. In other words, $\psi$ is a non-trivial dilation of the restriction $\pi|_A$. Therefore, by Proposition \ref{prop:unique-extn-iff-maximal}, $\pi$ does not have the unique extension property relative to $A$.
\end{proof}

The next result follows immediately from Theorem \ref{thm:translation-uep} and Theorem \ref{thm:uep-iff-boundary-measure}.

\begin{cor} \label{cor:translation-uep-pushforward}
Let $A$ be a concrete function system that generates a commutative C*-algebra $\rC(X)$, and let $\mu \in M^+(X)$ be a measure with corresponding GNS repre\-sen\-ta\-tion $\pi_\mu : \rC(X) \to \B(L^2(\mu))$. Let $K = S(A)$ denote the state space of $A$, let $\iota : A \to A(K)$ denote the canonical order isomorphism, and let $q : \rC(K) \to \rC(X)$ denote the canonical quotient map as in Theorem~$\ref{thm:function-systems}$. Then $\pi_\mu$ has the unique extension property relative to $A$ if and only if the pushforward measure $\mu \circ (q^*)^{-1} \in M^+(K)$ is maximal in the dilation order.
\end{cor}

\section{Hyperrigidity of function systems} \label{sec:hyperrigidity}

Motivated both by the fundamental role of the classical Choquet boundary in classical approximation theory, and by the importance of approximation in the contemporary theory of operator algebras, Arveson \cite{Arv2011} introduced hyperrigidity as a form of approximation that captures many important operator-algebraic phenomena.

\begin{defn}[Arveson]
A concrete operator system $A$ that generates a  C*-algebra $C$ is {\em hyperrigid} if whenever $\pi : C \to \B(H)$ is a nondegenerate $*$-representation and $\phi_n : C \to \B(H)$ is a sequence of unital completely positive maps with the property that
\[
 \lim_n \|\phi_n(a) - \pi(a)\| = 0 \qforal a \in A,
 \]
then
\[
 \lim_n \|\phi_n(c) - \pi(c)\| = 0 \qforal c \in C.
\]
\end{defn}

This definition is very useful once it is established, but there are equivalent formulations that are easier to verify.

\begin{thm}[Arveson] \label{T:Arv_hyper}
Let $A$ be a concrete operator system that generates a C*-algebra $C$. Then $A$ is hyperrigid if and only if every $*$-repre\-sen\-ta\-tion $\pi : C \to \B(H)$ has the unique extension property relative to $A$.
\end{thm}

It follows immediately from Definition \ref{defn:boundary-representation} and Theorem \ref{T:Arv_hyper} that a necessary condition for a concrete operator system to be hyperrigid is that every irreducible representation of the C*-algebra it generates is a boundary representation. Thus, in this case $\partial_A X = X$. Arveson conjectured \cite{Arv2011}*{Conjecture 4.3} that this is the only obstruction to hyperrigidity.

\begin{conj}[Arveson] \label{conj:hyperrigidity}
Let $A$ be a concrete operator system that generates a separable C*-algebra $C$. Then $A$ is hyperrigid if and only if every irreducible repre\-sen\-ta\-tion of $C$ is a boundary repre\-sen\-ta\-tion of $A$.
\end{conj}

A much more general problem than Conjecture \ref{conj:hyperrigidity} is to characterize the $*$-representations of a C*-algebra with the unique extension property relative to an operator system that generates it. Using the results of the previous section, we will provide a reformulation of this problem for the case of a function system.

\begin{question}
If $A$ is a concrete function system generating a separable commutative C*-algebra $\rC(X)$ such that $\partial_A X = X$, is every measure which is Choquet maximal (i.e. a boundary measure) also maximal in the dilation order? 
\end{question}

Note that a negative answer  yields a counterexample while a positive answer establishes the hyperrigidity conjecture in the  commutative case. 

More generally, if $\mu$ is a boundary measure for $A(K)$, we ask whether it is maximal in the dilation order.
If this has a positive answer, this would provide a precise analogue of Arveson's conjecture in the case when $\partial K$ is not closed.

It may be the case that separability is essential. 
Arveson's paper \cite{Arv2008} contains some discussion concerning the non-separable case, and why there may be more pathology there.

\section{Dilation versus Choquet order} \label{sec:dilation-vs-choquet}

It is unfortunately not true that the Choquet order and dilation order coincide. In this section we will show that the following example provides a counterexample.

\begin{example}\label{inequivalence}
Let $K$ be the convex hull of $\{e_1,e_2,e_3\}$ in $\bR^3$.
Let 
\[ x_i = (e_i + e_1 + e_2 + e_3)/4 \qand y=(e_1+e_2+e_3)/3. \]
Set $X = \{x_1,x_2,x_3\}$ and $Y = \{e_1,e_2,e_3,y\}$.
Set 
\[
 \mu = \frac13 \sum_{i=1}^3 \delta_{x_i} \qand  \nu_t = t \delta_y + \frac{1-t}3  \sum_{i=1}^3 \delta_{e_i} \qfor 0 \le t \le 1. 
\]
Then $\mu$ and each $\nu_t$ have barycenter $y$.

\begin{thm}
 $\mu \prec_c \nu_t$ if and only if $0 \le t \le \frac34$; but $\mu \prec_d \nu_t$ if and only if $0 \le t \le 0.9$.
\end{thm}

Let $A = A(a_1,a_2,a_3)$ denote the affine function such that $A(e_i)=a_i$.
Let $F \in C(K)$ satisfy $F(e_i)=F(x_i)=0$, $F(y)=1$ and $0 \le F \le 1$.

Then $H_\mu = L^2(\mu) = \bC^3$. Note that $C(X) = A(K)|_X$. 
Since we have $\pi_\mu(f) = \diag(f(x_1), f(x_2), f(x_3))$, it is easy to see that 
\[
 \pi_\mu(A(a,b,c)) = \diag\Big( \frac{2a+b+c}4, \frac{a+2b+c}4,\frac{a+b+2c}4 \Big) .
\]
Let $1_\mu = (1,1,1)$ and use the inner product $\ip{\xi,\eta} = \frac13 \sum_{i=1}^3 \xi_i \ol{\eta_i}$.

\begin{lem}
$\mu \prec_c \nu_t$ if and only if $0 \le t \le \frac34$.
\end{lem}

\begin{proof}
Since $\mu$ and $\nu_t$ have the same barycenter, they agree on affine functions.
So if $f \in P(K)$ is not affine, we may subtract an affine function so that $f(e_i) = 0$ for $1 \le i \le 3$, and scale it so that $f(y) = -1$.
Then by convexity, we have that
\[
f(x_i) = f( \tfrac14 e_i + \tfrac34 y) \le \tfrac14 f(e_i) + \tfrac34 f(y) = - \frac34 .
\]
Hence $\mu(f) \le - \frac34$. Moreover this is an equality for the function $f_0$ which is linear on the segments $[y,e_i]$.
However $\nu_t(f) = t f(y) = -t$.
Thus $\mu \prec_c \nu_t$ if and only if $- \frac34 \le -t$.
\end{proof}

\begin{lem}
$\mu \prec_d \nu_t$ if and only if $0 \le t \le 0.9$.
\end{lem}

\begin{proof}
We want to define a positive map $\Phi_t : C(K) \to \B(H_\mu)$ which factors through the quotient of $C(K)$ onto $C(Y)$
which agrees with $\pi_\mu$ on $A(K)$ and satisfies
\[ \ip{\Phi(f) 1_\mu, 1_\mu} = \nu_t(f)   \qfor f \in C(K) .\]
Observe that $C(Y)$ is spanned by $A(K)|_Y$ and $F$. 
As $\mu$ and $\nu_t$ have the same barycentre, they agree on $A(K)$.
Hence it suffices to arrange that
\[ t = \ip{\Phi(F) 1_\mu, 1_\mu} = \frac13 \sum_{i,j = 1}^3 \Phi(F)_{ij}.\]

Observe that in $C(Y)$, 
\begin{enumerate}
\item $A(a_1,a_2,a_3) +F \ge 0$ if and only if $a_i\ge 0$, and 
\item $A(a_1,a_2,a_3) - F \ge 0$ if and only if $a_i\ge0$ and $a_1+a_2+a_3 \ge 3$.
\end{enumerate}
So we require that
\begin{enumerate}
\item $\Phi_t(F) \ge 0$.
\item $\Phi_t(F)$ is less than $\pi_\mu(A(3,0,0)) = \diag(\tfrac32, \tfrac34, \tfrac34)$, and similarly less than $\diag(\tfrac34, \tfrac32, \tfrac34)$ and $\diag(\tfrac34, \tfrac34, \tfrac32)$.
\item $\sum_{i,j=1}^3 \Phi_t(F)_{ij} = 3t$.
\end{enumerate}

For $t\le 3/4$, we may take $\Phi_t(F) = t I_3$. This is the best that can be done mapping into the diagonal. 
However we can do better by setting 
\[ \Phi_t(F) = \frac t3  \begin{bmatrix}1&1&1\\1&1&1\\1&1&1\end{bmatrix} . \]
This evidently satisfies (1) and (3) for any $t \ge 0$.
By symmetry, to verify (2), it is enough to show that 
\[ 0 \le \begin{bmatrix}\frac32 - \frac t3&-\frac t3&-\frac t3\\-\frac t3&\frac34 - \frac t3&-\frac t3\\-\frac t3&-\frac t3&\frac34 - \frac t3\end{bmatrix} .\]
This has a positive diagonal and positive $2 \times 2$ minors when $t\le 1$. 
So we can detect when this ceases to be strictly positive by finding when it has kernel.
Gaussian elimination shows that this occurs when $t=.9$. 
\end{proof}

Therefore for $.75 < t \le .9$, we have 
\[
 \mu \prec_d \nu_t \quad\text{but}\quad \mu \not\prec_c \nu_t .
\]
This establishes that the two orders are distinct.
\medbreak

Let $f\in C(K)$. A familiar tool in Choquet theory is to define the upper concave envelope by
\[ \bar f = \inf \{ a \in A(K) : f \le a \} .\]
Now if we are interested in the maps $\Phi_t$, we need only consider $f_0 := f|_Y$.
That is, 
\[ \bar f_0 := \inf \{ \bar f : f|_Y = f_0 \} = \inf \{ a \in A(K) : f_0 \le a|_Y \} .\]

\begin{lem}
Let $f_0(e_i) = \alpha_i$ and $f_0(y) = \beta$. Set $\bar\alpha := \frac13 (\alpha_1 + \alpha_2 + \alpha_3)$. Then
\[ 
\bar f_0(x_i) = 
\begin{cases}
\frac{\alpha_i + \alpha_1 + \alpha_2 + \alpha_3}4 &\qif \beta \le \bar\alpha \\
\frac{\alpha_i + 3 \beta}4 &\qif \beta \ge \bar\alpha\end{cases}.
\]
\end{lem}

\begin{proof}
Observe that $A(a_1,a_2,a_3)|_Y \ge f_0$ provided that $a_i \ge \alpha_i$ and $ \frac13 (a_1 + a_2 + a_3) \ge \beta$.
So if $\beta \le  \bar\alpha$, then $\bar f_0 = A(\alpha_1,\alpha_2,\alpha_3)$.
Otherwise, one can use the extreme points of the minimal affine functions dominating $f_0$ to get that $\bar f_0$ equals
\[
 \min\!\big\{ \!
 A(\alpha_1,\alpha_2,3\beta \!-\! \alpha_1 \!-\! \alpha_2), \!
 A(\alpha_1,3\beta \!-\! \alpha_1 \!-\! \alpha_3,\alpha_3), \!
 A(3\beta \!-\! \alpha_2 \!-\! \alpha_3,\alpha_2,\alpha_3) 
 \! \big\} .
\]
In particular, in this case,
\[
 \bar f_0(x_i) 
 = \min\Big\{ \frac{\alpha_i + 3\beta}4, \frac{3\beta + \alpha_i + 3(\beta -  \bar\alpha)}4 \Big\}
 = \frac{\alpha_i + 3\beta}4.
\]
\end{proof}

\begin{cor}
For all $f\in C(K)$, $\Phi_t(f) \le \pi_\mu(\bar f)$ for $0 \le t \le .75$, but not for $.75 < t \le .9$.
\end{cor}

\begin{proof}
We have that 
\[
 \Phi_t(f) = \pi_\mu(A(\alpha_1,\alpha_2,\alpha_3)) + (\beta - \bar\alpha) tP
\]
where $P = \frac13 \begin{bmatrix}1&1&1\\1&1&1\\1&1&1\end{bmatrix}$.
As $\Phi_t$ depends only on $f_0=f|_Y$, we use the Lemma to get
\[
\pi_\mu(\bar f_0) = 
\begin{cases}
\pi_\mu(A(\alpha_1,\alpha_2,\alpha_3))&\ \text{if } \beta \le \bar\alpha \\
\diag\Big( \frac{\alpha_i + 3 \beta}4 \Big) = \pi_\mu(A(\alpha_1,\alpha_2,\alpha_3)) + \frac34(\beta-\bar\alpha) I_3&\ \text{if } \beta \ge \bar\alpha
\end{cases}.
\]
Subtracting, we obtain that
\[
\pi_\mu(\bar f_0) - \Phi_t(f) =
\begin{cases}
(\bar\alpha-\beta) tP &\qif \beta \le \bar\alpha \\
(\beta-\bar\alpha) (\frac34 I_3 - tP)&\qif \beta \ge \bar\alpha
\end{cases}.
\]
The first option is always positive, but the second is positive only if $tP \le \frac34 I_3$, which happens only if $t\le .75$.
\end{proof}
 
\end{example}


\end{document}